\documentclass[11pt,a4paper]{article}
\usepackage[latin1]{inputenc}
\usepackage{amsmath,amscd}
\usepackage{amssymb,latexsym,amsthm}
\usepackage{amsfonts}
\usepackage{color}
\usepackage[spanish,english]{babel}
\usepackage{makeidx}
\usepackage{graphicx}

\usepackage{amsmath,amsthm,amscd}
\usepackage[latin1]{inputenc}
\usepackage{lscape}
\usepackage{fancyhdr}
\usepackage{amsfonts}
\usepackage{pb-diagram}

\usepackage[all]{xy}

\usepackage{multicol}
\usepackage{eufrak}
\usepackage{amscd}
\usepackage{mathrsfs}

\usepackage[pdftex,colorlinks,a4paper,bookmarks=true,bookmarksnumbered=true,linkcolor=blue,citecolor=red]{hyperref}

\numberwithin{equation}{section}
\newtheorem{theorem}{Theorem}[section]

\newtheorem{definition}[theorem]{Definition}
\newtheorem{proposition}[theorem]{Proposition}

\newtheorem{conjecture}[theorem]{Conjecture}
\newtheorem{corollary}[theorem]{Corollary}

\theoremstyle{definition}

\newtheorem{example}[theorem]{Example}

\newtheorem{remark}[theorem]{Remark}

\def\N{\mathbb{N}}

\def\Z{\mathbb{Z}}
\def\Zn{\mathbb{Z}^n}

\def\Re{\mathbb{R}}
\def\Rn{\mathbb{R}^{n}}

\def\O{\mathcal{O}}
\def\q{\textbf{\textit{q}}}

\def\Oq{\O_{\q}}

\def\M{\mathscr{M}}

\def\F{\mathscr{F}}

\def\m{\mathcal{\textbf{m}}}

\def\Rq{R_{\q,\sigma}[x_1^{\pm 1},\ldots,x_r^{\pm 1},x_{r+1},\ldots, x_n]}
\def\Qrn{Q^{r,n}_{\q,\sigma}}
\def\Qnn{Q^{n,n}_{\q,\sigma}}
\def\Qnnq{Q^{n,n}_{\widetilde{\q},\widetilde{\sigma}}}

\title{\textbf{VALUATIONS OF SKEW QUANTUM POLYNOMIALS}}
\author{Cristian Arturo Chaparro Acosta\\
\texttt{crachaparroac@unal.edu.co}\\
\footnotesize Seminario de Álgebra Constructiva - SAC$^2$.\\
\footnotesize Departamento de Matemáticas.\\
\footnotesize Universidad Nacional de Colombia.\\
 \footnotesize Sede Bogotá.}
\date{}
\begin{document}
\maketitle
\begin{abstract}
In this paper we extend some results obtained by Artamonov and Sabitov for quantum polynomials to skew quantum polynomials and quasi$-$commutative bijective skew PBW extensions. Moreover, we find a counterexample to the conjecture proposed in \cite{A}.\\


{\footnotesize \textbf{Keywords:} Skew $PBW$ extensions, skew quantum polynomials, Ore domains, valuations, completions.}
\end{abstract}

\tableofcontents

\newpage
\section{Introduction}\label{intro}
This section is divided into three subsections, we recall the definition of $\Gamma-$valuation, valuation and quantum polynomials. We review some fundamental properties of valuations and valuations of quantum polynomials (see \cite{A4} and \cite{A}).

\subsection{Valuations}\label{val}
Let $D$ be a division ring, $D^*$ the multiplicative group and $\Gamma$ is a totally ordered group (with additive notation and not necessarily commutative).

\begin{definition}
A function $\nu:D^*\rightarrow \Gamma$ is a $\Gamma-$valuation of $D^*$ if:
\begin{enumerate}
	\item[i)] $\nu$ is surjective,

	\item[ii)] $\nu\left(ab\right)=\nu\left(a\right)+\nu\left(b\right)$,

	\item[iii)] $\nu\left(a+b\right)\geq \min \left\{ \nu\left(a\right),\nu\left(b\right)\right\} $.
\end{enumerate}
\end{definition}


\begin{proposition}\cite{S,F2}\label{propGV}
If $\nu$ is a $\Gamma-$valuation of $D^*$, then:
\begin{enumerate}
\item[1)] If $\nu\left(a\right)\neq\nu\left(b\right)$, then $\nu \left( a+b \right)=\min$$\left\{\nu\left(a\right),\nu\left(b\right)\right\} $.

\item[2)] $\Lambda_{\nu}:=\left\{ a\in D;\, a=0\,\text{ or }\,\nu\left(a\right)\geq0\right\} $ is a subring of $D$.

\item[3)] The group of units $\mathcal{U}_{\nu}:=\left\{ a\in D^{*};\nu\left(a\right)=0\right\} $ is a subgroup of $D^*$.

\item[4)] $\mathcal{W}_{\nu}:=\left\{ a\in D,a=0\,\text{ or }\,\nu\left(a\right)>0\right\} $ is a completely prime ideal of $\Lambda_{\nu}$ and $\mathcal{W}_{\nu}=\Lambda_{\nu}-\mathcal{U}_{\nu}$.

\item[5)] $\Lambda_{\nu}$ is a local ring with unique maximal ideal $\mathcal{W}_{\nu}$.
\end{enumerate}
\end{proposition}

\subsection{Valuations with values on $\Gamma\cup\{\infty \}$}

\begin{proposition}
Let $\Gamma$ be a totally ordered group with additive notation ordere. Then $\Gamma\cup \{\infty \}$ is an ordered additive monoid such that
$$x+\infty:=\infty=:\infty+x,\, \text{ for all } \Gamma\cup\{\infty\},$$ 
and $\infty>x$ for all $x\in \Gamma$.
\end{proposition}

\begin{definition}[\cite{Ch2}]
Let $R$ be a ring. By a valuation on $R$ with values in a totally ordered group $\Gamma$, the value group, we shall understand a function $\nu$ on $R$ with values in $\Gamma\cup\{\infty \}$ subject to the conditions:
\begin{enumerate}
	\item[i)] $\nu(a)\in\Gamma\cup\{\infty \}$ and $\nu$ assumes at least two values,

	\item[ii)] $\nu\left(ab\right)=\nu\left(a\right)+\nu\left(b\right)$,

	\item[iii)] $\nu\left(a+b\right)\geq \min \left\{ \nu\left(a\right),\nu\left(b\right)\right\} $.
\end{enumerate}
\end{definition}



\begin{proposition}\cite{Ch2, F2}
If $\nu$ is a valuation of $R$, then:
\begin{enumerate}
\item[1)] $ker\, \nu:=\{a\in R; \nu(a)=\infty\}$ is an ideal of $R$.
\item[2)] If $\nu(a+b)>\min\{\nu(a),\nu(b)\}$, then $\nu(a)=\nu(b)$.

\item[3)] $\Lambda_{\nu}:=\left\{ a\in R;\, \nu\left(a\right)\geq0\right\} $ is a subring of $R$.

\item[4)] The group of units $\mathcal{U}_{\nu}:=\left\{ a\in R^{*};\nu\left(a\right)=0\right\} $ is a subgroup of $R^*$.

\item[5)] $\mathcal{W}_{\nu}:=\left\{ a\in R,\nu\left(a\right)>0\right\} $ is an ideal of $\Lambda_{\nu}$.

\item[6)] $ker\, \nu$ is a completely prime ideal of $R$ and $R/ker\,\nu$ is an integral domain.
\end{enumerate}
\end{proposition}

\begin{proposition}[\cite{Ch2}]
If $\nu$ is a $\Gamma-$valuation of $D$. Then $\Gamma$ is abelian, if and only if $\nu(a)=0$ for all $a \in [D^*,D^*]$.
\end{proposition}

\subsection{Quantum polynomials}\label{ValQp}
Let $D$ be a division ring with a fixed set $\alpha_1$, $\alpha_2$ , $\ldots$, $\alpha_n$, $n \geq 2$, of automorphimsms. Also, we have $q_{ij}$ in $D^*$ for $i,j=1,2,\ldots ,n$ fix elements, satisfying the relations :
	\begin{gather*}
			q_{ii}=	q_{ij}q_{ji}=	\q_{ijr}	\q_{jri}	\q_{rij}=1\\
				\alpha_i(\alpha_j(d))=q_{ij}\alpha_j(\alpha_i(d))q_{ji},
	\end{gather*}
	
	where $\q_{ijr}=q_{ij}\alpha_j(q_{ir})$ and $d\in D$. We set $\q=(q_{ij})\in \M(n,D)$ and $\alpha=(\alpha_1,\alpha_2,$ $\ldots,\alpha_n)$.

\begin{definition}
The elements $q_{ij}$ of the matrix $\q$ are called \textbf{system of multiparameters}. 
\end{definition}

\begin{definition}[Quantum polynomial ring]
Denote by
\begin{equation}\label{pol}
\Oq:=D_{\q,\alpha}\left[x^{\pm 1}_1,x^{\pm1}_2,\ldots,x^{\pm 1}_r,x_{r+1},\ldots,x_n\right],
\end{equation}
 the associative ring generated by $D$ and by elements $x^{\pm 1}_1$, $x^{\pm 1}_2$, $\ldots$, $x^{\pm 1}_r$,  $x_{r+1}$, $\ldots$, $x_n$  subject to the defining relations
	\begin{gather}
		x_ix_i^{-1}=x_i^{-1}x_i=1\text{, } 1\leq i \leq r,\label{pol1}\\
		x_id=\alpha_i(d)x_i\text{, } d\in D\text{, } i=1,2,\ldots,n, \label{pol2}\\
		x_ix_j=q_{ij}x_jx_i\text{, } i,j=1,2,\ldots,n. \label{pol3}
	\end{gather}
\end{definition}

\begin{definition}
Let $N$ be the subgroup in the multiplicative group $D^*$ of the ring $D$ generated by the derived subgroup $[D^*,D^*]$ and by the set of all elements of the form $z^{-1}\sigma_i(z)$ where $z\in R^*$ and $i=1, \ldots, n$. $\Lambda:=D_{\q,\alpha}\left[x_1,x_2,\ldots,\ldots,x_n\right]$ is a general (generic) quantum polynomials ring if the images of all multiparameters $q_{ij}$, $1\leq i<j\leq n$, are independent in the multiplicative Abelian group $D^*/N$.
\end{definition}

The ring $\Oq$ is a left and right Noetherian domain, it satisfies Ore Condition and it has a division ring of fractions $F:=D_{\q}(x_1,\ldots,x_n)$. We consider $\nu:F^*\rightarrow \Gamma$ a $\Gamma-$valuation with $\nu(D^*)=0$.

\begin{theorem}[\cite{A}]\label{vs}
A valuation of a quantum division ring $D$, is Abelian in the sense that the group  $\Gamma$ is Abelian.
\end{theorem}

\begin{definition}[\cite{A4}, \cite{A}]
Let $\nu_{1}:D^{*}\rightarrow\Gamma_{1}$ and $\nu_{2}:D^{*}\rightarrow\Gamma_{2}$ be two  valuations. Set $\nu_{1}\geq\nu_{2}$ if there exists an epimorphism of ordered groups $\tau:\Gamma_{1}\rightarrow\Gamma_{2}$ such that $\tau\nu_{1}=\nu_{2}$. It means that the diagram

$$
\xymatrix{
D^* \ar[d]_{\nu_2} \ar[r]^{\nu_1} & \varGamma_1 \ar[dl]^{\tau} \\
\varGamma_2&}
$$
is commutative.
\end{definition}

\begin{definition}[\cite{A4}, \cite{A}]
A valuation $\nu$ has a maximal rank if $\tau$ is an isomorphism in the previous definition.
\end{definition}

\begin{theorem}[\cite{A4}]\label{maximal}
A valuation $\nu:F^*\rightarrow \Gamma$  of a general quantum division ring $\O_q$ is has maximal rank if only if $\Gamma\cong\Zn$.
\end{theorem}

\section{Completions of quantum polynomials}\label{CompQp}

In this section $\nu:F^*\rightarrow \Zn$  is a maximal $\Zn-$valuation.

\begin{definition}[\cite{A}]

Let $\F$ be the set of all maps $f:\Zn\rightarrow k$ and the zero element such that $supp\, f:=\{m\in\Zn; f(m)\neq 0\}$ is Artinian with respect to the lexicographic order on $\Zn$.
\end{definition}

\begin{theorem}
$\F$ is a division ring containing $F$.
\end{theorem}
\begin{proof}
See $\cite{A3}$ Theorem $3.4$ and $3.7$.
\end{proof}

Expand the valuation $\nu$ to $f\in\F$ in the following way. If $f \in \F$ then $\nu(f)$ the least element from $supp\,f$.

\begin{definition}[\cite{A}]
The division ring $\F$ is called a \textit{completion} of $F$ with respect to $\nu$. 
\end{definition}

\begin{remark}
If $\O:=\{f\in \F;\nu(f)\geq 0 \}$ and $\m:=\{f\in \F;\nu(f)>0\}$, then $\O$ is a subring in $\F$ and $\m$ is an ideal in $\O$. Moreover, $\O/\m \cong k$.
\end{remark}

Let $\Rn$ be a vector space of all rows $(r_1,\ldots , r_n)$, $r_i \in \Re$, of a length $n$ and $\Rn$ is equipped with the lexicographic order.

\begin{theorem}[\cite{K16}]
Let $\leq_{\Zn}$ be a totally order in the additive group $\Zn$. Then there exists order preserving group embedding $\Zn\rightarrow \Rn$.
\end{theorem}

\begin{definition}\cite{A}
A totally order $\leq_{\Zn}$ is essentially lexicographic if it belongs to the orbit of the standard embedding of $\Zn$ in to $\Rn$ under the action of the group  $GL(n, \Z)$. \textit{i.e.}, if $a,b\in \Zn$, $a\leq_{\Zn} b$ if only if $aA\leq bA$ for some fixed $A$ in $ GL(n,\Z)$ and $\leq$ the lexicographic order.
\end{definition}

\begin{conjecture}[\cite{A}]
A valuation $\nu$ is associated to an essentially lexicographic order on $\Zn$ if and only if $\cap_{i>1}\m ^i = 0$.
\end{conjecture}

In the study of this conjecture we obtain the following results partial:

\begin{proposition}\label{ArqInf}
If $\nu:R\rightarrow \Gamma \cup \{\infty \}$ is a valuation of a ring $R$ and $\Gamma$ is a Archimedean group with $\mathcal{W}_{\nu}:= \left\{ a\in R,\nu\left(a\right)>0\right\} $, $\inf\{\nu(\mathcal{W}_{\nu})\}\neq 0$ and $\bigcap_{i\geq1} \mathcal{W}_{\nu} ^i:=I$, then $\nu(I)=\infty$.
\end{proposition}
\begin{proof}
Let $A_i:=\nu(\mathcal{W}_{\nu}^i)$ and $\lambda_i:=\inf\{A_i\}$ be, then $\lambda_1<\lambda_2<...<\lambda_i$ and $i\lambda_1\leq\lambda_i$, indeed: (by induction over $i$) as $\inf\{\nu(\mathcal{W}_{\nu})\}\neq 0$ then $0<\lambda_1\leq\nu(a)$ for all $a\in\mathcal{W}_{\nu}$, hence $\lambda_1<2\lambda_1\leq\nu(ab)$ for all $a,b\in \mathcal{W}_{\nu}$, therefore $2\lambda_1\leq\lambda_2$, suppose that $\lambda_{i-1}<\lambda_i$ and $i\lambda_1\leq \lambda_i$, then $i\lambda_1< (i+1)\lambda_1\leq\lambda_i +\lambda_1\leq\nu(a)+\nu(b)=\nu(ab)$ for all $a\in\mathcal{W}_{\nu}^i$  and  $b\in\mathcal{W}_{\nu}$, then, $\lambda_{i}<\lambda_{i+1}$ and $(i+1)\lambda_1\leq \lambda_{i+1}$.\\

Now, suppose there exists $b\in I$ such that $\nu(b)=\lambda<\infty$, so $\lambda_1<\lambda$ and as $\Gamma$ is Archimedean  there exists an integer $m$ such that $m\lambda_1>\lambda$, therefore $\lambda\notin A_m$,  hence $b\notin \mathcal{W}_{\nu}^m$, contradicting that $b\in I$.
\end{proof}

\begin{corollary}\label{coroD}
If $\nu:D\rightarrow \Gamma \cup \{\infty \}$ is a valuation of a division ring $D$ and $\Gamma$ is a Archimedean group with $\inf\{\nu(\mathcal{W}_{\nu})\}\neq 0$, then $\bigcap_{i\geq1} \mathcal{W}_{\nu} ^i=0$.
\end{corollary}
\begin{proof}
$0=\nu(1)=\nu(aa^{-1})=\nu(a)+\nu(a^{-1})$ for all $a\in D^*$, then $\nu(a)<\infty$ for all $a\in D^*$, therefore $\nu(a)=\infty$ if only if $a=0$.
\end{proof}
\begin{remark}
In the Proposition \ref{ArqInf} the condition  $\inf\{\nu(\mathcal{W}_{\nu})\}\neq 0$ can be replaced by $\inf\{\nu(\mathcal{W}^i_{\nu})\}\neq 0$ for any $i>0$ in $\N$.
\end{remark}

\begin{example}
If we take lexicographic order on $\Z^2$ the order does not have intersection property: consider $A:=\{(x,y)\in\Z^2; (0,0)<(x,y) \}$ and $nA:=\sum_{i=1}^{n}A$ with $n>0$ , then $nA=\{(x,y)\in\Z^2; (0,n)\leq(x,y)\}$.
By induction over n: If $n=2$, then $2A=A	\setminus\{(0,1)\}$, indeed: as $\min\{A\}=(0,1)$ then $(0,2)\leq (x,y)$ with $(x,y)\in 2A$, thus $2A\subseteq A	\setminus\{(0,1)\}$. Now, if $(x,y)$ in $2A$, then $(x,y-1)\in A$, because $x>0$ or $x=0$ and $y\geq 2$.\\

Suppose that $nA=(n-1)A	\setminus\{(0,n-1)\}$, as $\min\{nA\}=(0,n)$ then $(0,n+1)\leq (x,y)$ with $(x,y)\in (n+1)A$, thus $(n+1)A\subseteq nA	\setminus\{(0,n)\}$. Now, if $(x,y)$ in $(n+1)A$, then $(x,y-1)\in nA$, because $x>0$ or $x=0$ and $y\geq n+1$. Consequently $(n+1)A=\{(x,y)\in\Z^2; (0,n+1)\leq(x,y)\}$.\\

Hence, as $(1,0)\in nA$ for every $n\geq 1$ since $(0,n)<(1,0)$ , then $(1,0)\in \cap_{n>0} nA$.
\end{example}

It follows a counterexample to the conjecture, since a lexicographic order is essentially lexicographic.


\section{Skew $PBW$ extensions}\label{definitionSPBW}

In this section we recall the definition and some basic properties of skew PBW (Poincaré-Birkhoff-Witt) extensions, introduced in \cite{L2}. Some ring-theoretic and homological properties of these class of noncommutative rings have been studied in \cite{L}.\\

 \begin{definition}\label{gpbwextension}
Let $R$ and $A$ be rings. We say that $A$ is a \textit{skew $PBW$ extension of $R$} $($also called a $\sigma-PBW$ extension of $R$$)$ if the following conditions hold:
	\begin{enumerate}
		\item[\rm (i)]$R\subseteq A$.
		\item[\rm (ii)] There exists finitely many elements $x_1,\dots ,x_n\in A$ such $A$ is a left $R$-free module with basis
	
		\begin{center}
	
		${\rm Mon}(A):= \{x^{u}=x_1^{u_1}\cdots x_n^{u_n}\mid u=(u_1,\dots ,u_n)\in \mathbb{N}^n\}$.
		\end{center}

		In this case it also says that \textit{$A$ is a left polynomial ring over $R$} with respect to $\{x_1,\dots,x_n\}$ and $Mon(A)$ is the set of standard monomials of $A$. Moreover, $x_1^0\cdots x_n^0:=1\in Mon(A)$. 
		
		\item[\rm (iii)]For every $1\leq i\leq n$ and $r\in R-\{0\}$ there exists $c_{i,r}\in R-\{0\}$ such that
		\begin{equation}\label{sigmadefinicion1}
			x_ir-c_{i,r}x_i\in R.
		\end{equation}
		
		\item[\rm (iv)]For every $1\leq i,j\leq n$ there exists $c_{i,j}\in R-\{0\}$ such that
		\begin{equation}\label{sigmadefinicion2}
			x_jx_i-c_{i,j}x_ix_j\in R+Rx_1+\cdots +Rx_n.
		\end{equation}
		Under these conditions we will write $A:=\sigma(R)\langle x_1,\dots ,x_n\rangle$.
	\end{enumerate}
\end{definition}

\begin{proposition}\label{sigmadefinition}
Let $A$ be a skew $PBW$ extension of $R$. Then, for every $1\leq i\leq n$, there exists an injective ring endomorphism $\sigma_i:R\rightarrow R$ and a $\sigma_i$-derivation $\delta_i:R\rightarrow R$ such that
	\begin{center}
		$x_ir=\sigma_i(r)x_i+\delta_i(r)$,
	\end{center}
	for each $r\in R$.
\end{proposition}
\begin{proof}
See \cite{L2}, Proposition 3.
\end{proof}

The previous proposition gives the notation and the alternative name given for the skew $PBW$ extensions. 

\begin{definition}\label{sigmapbwderivationtype}
Let $A$ be a skew $PBW$ extension.
	\begin{enumerate}
		\item[\rm (a)] $A$ is quasi-commutative if the conditions {\rm(}iii{\rm)} and {\rm(}iv{\rm)} in Definition \ref{gpbwextension} are replaced by
		
		\begin{enumerate}
			\item[\rm (iii')] For every $1\leq i\leq n$ and $r\in R-\{0\}$ there exists $c_{i,r}\in R-\{0\}$ such that
			\begin{equation}
				x_ir=c_{i,r}x_i.
			\end{equation}
			
		\item[\rm (iv')]For every $1\leq i,j\leq n$ there exists $c_{i,j}\in R-\{0\}$ such that
			\begin{equation}
				x_jx_i=c_{i,j}x_ix_j.
			\end{equation}
		\end{enumerate}
		
		\item[\rm (b)]$A$ is bijective if $\sigma_i$ is bijective for every $1\leq i\leq n$ and $c_{i,j}$ is invertible for any $1\leq i<j\leq n$.
	\end{enumerate}
\end{definition}

\begin{definition}\label{1.1.6}
Let $A$ be a skew $PBW$ extension of $R$ with endomorphisms $\sigma_i$, $1\leq i\leq n$, as in Proposition \ref{sigmadefinition}.

\begin{enumerate}
	\item[\rm (i)] For $u=(u_1,\dots,u_n)\in \mathbb{N}^n$, $\sigma^{u}:=\sigma_1^{u_1}\cdots \sigma_n^{u_n}$, $|u|:=u_1+\cdots+u_n$. If $v=(v_1,\dots,v_n)\in \mathbb{N}^n$, then $u+v:=(u_1+v_1,\dots,u_n+v_n)$. 
	\item[\rm (ii)]For $X=x^{u}\in {\rm Mon}(A)$, $\exp(X):=u$ and $\deg(X):=|u|$.
	\item[\rm (iii)]If $f=c_1X_1+\cdots +c_tX_t$, with $X_i\in Mon(A)$ and $c_i\in R-\{0\}$, then $\deg(f):=\max\{\deg(X_i)\}_{i=1}^t.$
\end{enumerate}
\end{definition}

\begin{theorem}\label{coefficientes}
Let $A$ be a left polynomial ring over $R$ w.r.t. $\{x_1,\dots,x_n\}$. $A$ is a skew $PBW$ extension of $R$ if and only if the following conditions hold:
	\begin{enumerate}
		\item[\rm (a)]For every $x^{u}\in {\rm Mon}(A)$ and every $0\neq r\in R$ there exist unique elements $r_{u}:=\sigma^{u}(r)\in R-\{0\}$ and $p_{u ,r}\in A$ such that
			\begin{equation}\label{611}
				x^{u}r=r_{u}x^{u}+p_{u , r},
			\end{equation}
		where $p_{u ,r}=0$ or $\deg(p_{u ,r})<|u|$ if $p_{u , r}\neq 0$. Moreover, if $r$ is left invertible, then $r_u$ is left invertible.

		\item[\rm (b)]For every $x^{u},x^{v}\in {\rm Mon}(A)$ there exist unique elements $c_{u,v}\in R$ and $p_{u,v}\in A$ such that
			\begin{equation}\label{612}
				x^{u}x^{v}=c_{u,v}x^{u+v}+p_{u,v},
			\end{equation}
		where $c_{u,v}$ is left invertible, $p_{u,v}=0$ or $\deg(p_{u,v})<|u+v|$ if $p_{u,v}\neq 0$.
	\end{enumerate}
\begin{proof}
See \cite{L2}, Theorem 7.
\end{proof}
\end{theorem}

\begin{proposition}\label{dominio}
Let A be a skew PBW extension of a ring R. If R is a domain, then A is a domain.
\end{proposition}
\begin{proof}
See \cite{L}.
\end{proof}

The next theorem characterizes the quasi-commutative skew $PBW$ extensions.

\begin{theorem}\label{iterados}
Let $A$ be a quasi-commutative skew $PBW$ extension of a ring $R$. Then,
	\begin{enumerate}
		\item[\rm (i)] $A$ is isomorphic to an iterated skew polynomial ring of endomorphism type, i.e.,
		\begin{center}
			$A\cong R[z_1;\theta_1]\cdots [z_{n};\theta_n]$.
		\end{center}
		\item[\rm (ii)] If $A$ is bijective, then each endomorphism $\theta_i$ is bijective, $1\leq i\leq n$.
	\end{enumerate}
\end{theorem}
\begin{proof}
See \cite{L}.
\end{proof}
\begin{corollary}\label{oreccb}
Let $A$ be a bijective and quasi-commutative  skew $PBW$ extension of a ring $R$. If $R$ is a left Ore domain, then $A$ is a left Ore domain.
\end{corollary}
\begin{proof}
By Theorem \ref{iterados}, $A$ is isomorphic to an iterated skew polynomial ring of automorphism type over a left Ore domain $R$.
\end{proof}

\begin{theorem}\label{filtrado}
Let $A$ be an arbitrary skew $PBW$ extension of $R$. Then, $A$ is a filtered ring with filtration given by
	\begin{equation}\label{eq1.3.1a}
		F_m:=\begin{cases} R & {\rm if}\ \ m=0\\ \{f\in A\mid {\rm deg}(f)\le m\} & {\rm if}\ \ m\ge 1
	\end{cases}
	\end{equation}
and the corresponding graded ring $Gr(A)$ is a quasi-commutative skew $PBW$ extension of $R$.
Moreover, if $A$ is bijective, then $Gr(A)$ is a quasi-commutative bijective skew $PBW$ extension of $R$.
\end{theorem}
\begin{proof}
See \cite{L}.
\end{proof}

\begin{theorem}[Hilbert Basis Theorem]\label{hilbert}
Let $A$ be a bijective skew $PBW$ extension of $R$. If $R$ is a left $($right$)$ Noetherian ring then $A$ is also a left $($right$)$ Noetherian ring.
\end{theorem}
\begin{proof}
See \cite{L}.
\end{proof}

\subsection{Skew quantum polynomials}\label{definitionSQP}
In this subsection we recall the definition and some basic properties of skew quantum polynomials  ring over $R$, introduced in \cite{L}.We mention some results generalized for valuations of skew quantum polynomials and  bijective and quasi-commutative  skew $PBW$ extension.

\begin{definition}
Let $R$ be a ring with matrix of parameters $q:=[q_{ij}]\in M_n(R)$, $n\geq 2$, such that $q_{ii}=1=q_{ij}q_{ji}=q_{ji}q_{ij}$ for each $1\leq i,j \leq n$ and suppose also that is given a system $\sigma_1,\ldots, \sigma_n $ of automorphisms of $R$. The  \textit{skew quantum polynomials  ring over $R$}, denoted by

\begin{equation}
R_{\q,\sigma}[x_1^{\pm 1},\ldots,x_r^{\pm 1},x_{r+1},\ldots, x_n],
\end{equation}

is defined whit the following conditions:

\begin{enumerate}
	\item[i)] $R\subseteq R_{\q,\sigma}[x_1^{\pm 1},\ldots,x_r^{\pm 1},x_{r+1},\ldots, x_n],$
	
	\item[ii)] $R_{\q,\sigma}[x_1^{\pm 1},\ldots,x_r^{\pm 1},x_{r+1},\ldots, x_n]$ is a free left $R-$module with basis $\{x^u;x^u = x_1^{u_1}\cdots x_n^{u_n}, u_i\in \Z \text{, }1\leq i \leq r \text{ and } u_i\in \N \text{ for }r+1\leq i \leq n \},$

	\item[iii)] The $x_1,\ldots,x_n$ elements satisfy the defining relations
		\begin{gather}
			x_ix_i^{-1}=1=x_i^{-1}x_i, \text{ } 1\leq i \leq r,\label{polc1}\\
			x_ix_j=q_{ji}x_jx_i \text{ } 1 \leq i,j\leq n, \label{polc2}\\
			x_ir=\sigma_i(r)x_i, \text{ } r\in R\text{ y } 1 \leq i\leq n. \label{polc3}
		\end{gather}
\end{enumerate}
\end{definition}

When all automorphisms are trivial, we write $R_{\q}[x_1^{\pm 1},\ldots,x_r^{\pm 1},$ $x_{r+1},$ $\ldots, $ $x_n]$ and this ring is called  \textit{the ring of quantum polynomials over $R$.} If $R=K$ is a field, then $K_{\q,\sigma}[x_1^{\pm 1},\ldots,x_r^{\pm 1},x_{r+1},$ $\ldots, x_n]$ is the \textit{algebra of skew quantum polynomials}. For  trivial automorphisms we get the \textit{algebra of quantum polynomials simply}.\\

If $r=n$, $R_{\q,\sigma}[x_1^{\pm 1},\ldots,x_n^{\pm 1}]$ is called the \textit{n-multiparametric skew quantum torus over $R$}, when all automorphisms are trivial, is called the \textit{$n-$multipara\-metric quantum torus over $R$}. If $r=0$, $R_{\q,\sigma}[x_1,\ldots, x_n]$ is called the \textit{n-multiparametric skew quantum space over $R$}, when all automorphisms are trivial is called \textit{ $n-$multiparametric quantum space over $R$}.\\

The algebra of quantum polynomials can be defined as a quasi-commutative bijective skew $PBW$ extension of the $r$-multiparameter quantum torus, or also, as a localization of a
quasi-commutative bijective skew $PBW$ extension.

\begin{theorem}\label{tpct1}
$R_{\q,\sigma}[x_1,\ldots,x_n]\cong R[z_1;\theta_1]\cdots[z_n;\theta_n]$, where
\begin{equation*}
	\begin{array}{l}
		\text{ }\text{ }i) \text{ }\theta_1=\sigma_1, \\
		\text{ }ii) \text{ } \theta_i:R[z_1;\theta_1]\cdots[z_{i-1};\theta_{i-1}]\rightarrow R[z_1;\theta_1]\cdots[z_{i-1};\theta_{i-1}], \\
		iii) \text{ }\theta_i(z_i)=q_{ij}z_i, 1\leq i<j\leq n, \theta_i(r)=\sigma_i(r)\text{ for } r\in R.
	\end{array}
	\end{equation*}
In particular,  $R_{\q}[x_1,\ldots,x_n]\cong R[z_1]\cdots[z_n;\theta_n]$.
\end{theorem}
\begin{proof}
See \cite{L}.
\end{proof}

\begin{theorem}\label{tpct2}
$R_{\q,\sigma}[x_1^{\pm 1},\ldots,x_r^{\pm 1},x_{r+1}\ldots,x_n]$ is a ring of fractions of $B:=R_{\q,\sigma}[x_1,$ $\ldots,x_n]$ with respect to the multiplicative subset $$S=\{rx^u; r\in R^*, x^u\in Mon\{x_1,\ldots,x_r\}\},$$ i.e, 
$$R_{\q,\sigma}[x_1^{\pm 1},\ldots,x_r^{\pm 1},x_{r+1}\ldots,x_n]\cong S^{-1}B.$$
\end{theorem}
\begin{proof}
See \cite{L}.
\end{proof}

\begin{remark}\label{pctq}
Let $Q^{r,n}_{\q,\sigma}(R):= \Rq$ and $R$ be a left (right) Noetherian ring, then $\Qrn(R)$ is left (right) Noetherian by Theorem \ref{hilbert}. Moreover, if $R$ is a domain, then $\Qrn(R)$ is also a domain by Theorem \ref{dominio}. Thus, if $R$ is a left (right) Noetherian domain, then $\Qrn(R)$ is a left (right) Ore domain.\\

Thus, $\Qrn(R)$ has a total division ring of fractions
$$Q(\Qrn(R))\cong Q(A):=\sigma(R)(x_1,\ldots,x_n),$$
where $\sigma(R)(x_1,\ldots,x_n)$ denotes the rational fractions of $A:=\sigma(R)\langle x_1,\ldots,x_n\rangle$.
\end{remark}

\subsection{Some properties}

\begin{definition}
Let $N$ be the subgroup in the multiplicative group $R^*$ of the ring $R$ generated by the derived subgroup $[R^*,R^*]$ and by the set of all elements of the form $z^{-1}\sigma_i(z)$ where $z\in R^*$ and $i=1, \ldots, n$.
\end{definition}

\begin{remark}\label{CCN}
$N$ is a normal subgroup in $R^*$.
\end{remark}

\begin{definition}
If the images of $q_{ij}$ with $1\leq i <j\leq n$ are independent in the  multiplicative Abelian group ${\bar{R}}=R^*/N$ then, $R_{\q,\sigma}[x_1^{\pm 1},\ldots,x_r^{\pm 1},x_{r+1},$ $\ldots, x_n]$ is a generic  skew quantum polynomials ring.
\end{definition}

\begin{remark}
If n=2 in $R_{\q,\sigma}[x_1^{\pm 1},\ldots,x_r^{\pm 1},x_{r+1},\ldots, x_n]$, of the previous definition  $q=q_{12}$ is not a root of unity.
\end{remark}

\begin{proposition}\label{aut} 
For each  $a\in R^*$ and $\sigma$ endomorphism over $R$, $\sigma^k(a)=an$ with $k\in \N$ and $n\in N$.
\end{proposition}
\begin{proof}
\begin{eqnarray}
\sigma^k(a)& = & a\left(a^{-1}\sigma(a)\right) \left( (\sigma(a))^{-1}\sigma^2(a) \right)\ldots \left((\sigma^{k-1}(a))^{-1} \sigma^k(a)\right) \nonumber\\
& = & a n, \text{ with } n\in N. \label{p4,1}
\end{eqnarray}
\end{proof}

\begin{proposition}
If $u, v \in \Z^r\times \N^{n-r}$ and $\lambda,\mu\in R^*$, then
\begin{enumerate}
	\item[(1)] $x_ix^u=\left( \prod_{j=1}^{n}q_{ji} ^{u_{j}}\right) n_u\cdot x^ux_i,$  for some $n_u\in N$ and  for any $1\leq i\leq n$.

	\item[(2)] $\left( x^u\right) \left(x^v\right)  =\left( \prod_{i<j}q_{ji} ^{u_{j}v_i}\right) n_{u+v}\cdot x^{u+v},$
	 with $n_{u+v}\in N$.
	 
	 \item[(3)] $\left( \lambda x^u\right) \left(  \mu x^v\right)  =\lambda\mu\left( \prod_{i<j}q_{ji} ^{u_{j}v_i}\right) n'\cdot x^{u+v},$
	 with $n'\in N$.
\end{enumerate}
\end{proposition}
\begin{proof}
Applying the Proposition \ref{aut} and note that $x_ix_j^{-1}=q_{ji}^{-1}x_j^{-1}x_i$ with $1\leq j \leq r$.
\end{proof}

\begin{proposition}
Let $ f:=\sum_{u\in\Z}\lambda_u x^u$ be in $ R_{\q,\sigma}[x_1^{\pm 1},\ldots,x_r^{\pm 1},x_{r+1},\ldots, x_n]$ and $ x_i$ with $1\leq i\leq r$.

\begin{enumerate}
	\item[(1)] If $ \lambda_u \in R$, then
	$$x_i f x_i^{-1}=\sum_{u\in\Z^n}\sigma_i(\lambda_u)\lambda'_ux^u,$$
	where $ \lambda'_u:=\left( \prod_{j=1}^n q_{ji} ^{u_j}\right) n_u\in R^*.$
	\item[(2)] If $\lambda_u \in R^*$, then $$x_i f x_i^{-1}=\sum_{u\in\Z^n}\lambda'_ux^u,$$
	where $ \lambda'_u\in R^*.$
\end{enumerate}
\end{proposition}
\begin{proof}
\begin{enumerate}
\item[(1)] Note that $N\subseteq R^*$ and
\begin{eqnarray*}
x_i f x_i^{-1} & = &  \sum\sigma_i(\lambda _u)x_i x^u x^{-i}\\
 & = & \sum_{u\in\Z^n}\sigma_{i}(\lambda _{u})\left( \prod_{j=1}^n q_{ji} ^{u_j}\right) n_ux^u,
\end{eqnarray*}
where $ n_u \in N$.
\item[(2)] By item (1), $\sigma_i(\lambda_u)\lambda'_u\in R^*$.
\end{enumerate}
\end{proof}
\begin{remark}
If $Q(Q_{\q,\sigma}^{r,n}(R))$ exists and $G$ denotes the multiplicative subgroup in $Q(Q_{\q,\sigma}^{r,n}(R))^{*}$ generated by $R^{*}$ and $x_{1},...,x_{n}$. Then $R^{*}\triangleleft G$ and $G/R^{*}$ is a  free abelian group with the base $x_{1}R^{*},\ldots,$ $x_{n}R^{*}$.
\end{remark}

\begin{proposition}\label{tpcaut}
Let $R$ be a left Ore domain and $\sigma$ automorphisms over $R$, then $\sigma$ can be extended to $Q(R)$ and is also an automorphism.
\end{proposition}
\begin{proof}
By universal property we have the following commutative diagram:
$$
\xymatrix{
R  \ar[d]_{\sigma} \ar[r]^{\psi} & Q(R) \ar@{-->}[2,-1]^{\widetilde{\sigma}} \\
R \ar[d]_{\psi}&\\
Q(R)}
$$
where $\psi, \sigma$ are injective and $\widetilde{\sigma}\left( \frac{a}{b}\right)=\frac{\sigma(a)}{\sigma(b)}$  for  $a,b\neq 0\in R$. Therefore, $\psi\circ \sigma$ is injective and so is $\widetilde{\sigma}$.\\

If $\frac{a}{b}\in Q(R)$, then $\frac{a}{b}=\psi(b)^{-1}\psi(a)=\psi(\sigma(b_0))^{-1}\psi(\sigma(a_0))$ for $a_0, b_0\neq 0 \in R$, consequently,
\begin{eqnarray*}
\frac{a}{b}&=&\psi(\sigma(b_0))^{-1}\psi(\sigma(a_0))\\
&=&\widetilde{\sigma}(\psi(b_0))^{-1}\widetilde{\sigma}(\psi(a_0))\\
&=&\widetilde{\sigma}(\psi(b_0)^{-1}\psi(a_0))\\
&=&\widetilde{\sigma}\left( \frac{a_0}{b_0}\right). 
\end{eqnarray*} 
\end{proof}

\begin{theorem}\label{pctaut}
Let $R$ be a left Ore domain and $S=R-\{0\}$, then 
$$S^{-1}(R_{\q,\sigma}[x_1,\ldots, x_n])\cong Q(R)_{\widetilde{\q},\widetilde{\sigma}}[x_1,\ldots, x_n],$$
where $\widetilde{\q}=\left(\frac{q_{ij}}{1}\right)\in \M(n,Q(R))$.
\end{theorem}
\begin{proof}
By Theorem \ref{tpct1} $R_{\q,\sigma}[x_1,\ldots, x_n]\cong R[z_1;\theta_1]\cdots[z_n;\theta_n]$, with each $\theta_i$ bijective. Thus, if $S=R-\{0\}$ then
\begin{eqnarray*}
S^{-1}\left( R_{\q,\sigma}[x_1,\ldots, x_n]\right) &\cong & S^{-1}\left( R[z_1;\theta_1]\cdots[z_n;\theta_n]\right) \\
&\cong& S^{-1}\left( R\right) [z_1;\widetilde{\theta_1}]\cdots[z_n;\widetilde{\theta_n}]\\
&=& Q\left( R\right) [z_1;\widetilde{\theta_1}]\cdots[z_n;\widetilde{\theta_n}]
\end{eqnarray*} 
where
\begin{eqnarray*}
\widetilde{\theta_1}: Q(R) &\rightarrow &Q(R)\\
 \frac{a}{b}  &\mapsto&\widetilde{\theta_1}\left(\frac{a}{b} \right)=\frac{\theta_1(a)}{\theta_1(b)}=\frac{\sigma_1(a)}{\sigma_1(b)}=\widetilde{\sigma_1}\left(\frac{a}{b} \right),
\end{eqnarray*} 
and
\begin{eqnarray*}
\widetilde{\theta_i}: Q\left( R\right) [z_1;\widetilde{\theta_1}]\cdots[z_{i-1};\widetilde{\theta_{i-1}}]&\rightarrow & Q\left( R\right) [z_1;\widetilde{\theta_1}]\cdots[z_{i-1};\widetilde{\theta_{i-1}}]
\end{eqnarray*}
with $$\widetilde{\theta_i}\left(\frac{a}{b} \right)=\widetilde{\sigma_i}\left( \frac{a}{b}\right) \text{ y } \widetilde{\theta_j}\left(z_i \right)=\frac{q_{ij}}{1}z_i.$$

Therefore, 
$$S^{-1}(R_{\q,\sigma}[x_1,\ldots, x_n])\cong Q(R)_{\widetilde{\q},\widetilde{\sigma}}[x_1,\ldots, x_n],$$
where $\widetilde{\q}=\left(\frac{q_{ij}}{1}\right)\in \M(n,Q(R))$.
\end{proof}

\begin{proposition}\label{injection}
Let $R$ be a left Ore domain ,  there exists $$\phi:R_{\q,\sigma}[x^{\pm1}_1,\ldots, x^{\pm1}_n]\rightarrow Q(R)_{\widetilde{\q},\widetilde{\sigma}}[x^{\pm1}_1,\ldots, x^{\pm1}_n]$$ an injective ring homomorphism.
\end{proposition}
\begin{proof}
Let $B_R:=R_{\q,\sigma}[x_1,\ldots, x_n]$ and $B_{Q(R)}:=Q(R)_{\widetilde{\q},\widetilde{\sigma}}[x_1,\ldots, x_n]$ be, by Theorem \ref{tpct2} $R_{\q,\sigma}[x^{\pm1}_1,$ $\ldots, x^{\pm1}_n]\cong S^{-1}_1B_R$ with $S_1=\{rx^u; r\in R^*,$ $ x^u\in Mon\{x_1,\ldots,x_n\} \}$,  and $Q(R)_{\widetilde{\q},\widetilde{\sigma}}[x^{\pm1}_1,$ $\ldots, x^{\pm1}_n]$ $\cong$ $S^{-1}_{1'}B_{Q(R)}$ with $S_{1'}=\{rx^u; r\in Q(R)^*, x^u \in Mon\{x_1,\ldots,x_n\} \}$.\\

Now, consider the following diagram of ring homomorphisms:

$$
\xymatrix{
R\ar[d]_{\psi} \ar@{^(->}[0,2]& & R_{\q,\sigma}[x_1,\ldots, x_n]\ar[d]_{\psi'}\ar[0,2] & \ar@{-}^{\psi_1}  & S^{-1}_1B_R  \ar@{-->}[d]_{\varphi}\\ 
Q(R) \ar@{^(->}[0,2] & & Q(R)_{\widetilde{\q},\widetilde{\sigma}}[x_1,\ldots, x_n] \ar[0,2] & \ar@{-}^{\psi_{1'}} & S^{-1}_{1'}B_{Q(R)}  
}
$$
where $\psi$ is the  injection for the  localization of $R$ to the total ring fractions $Q(R)$, $\psi'$ the injection determined by the isomorphism of Theorem \ref{pctaut} where $\psi'(ax^u)=\frac{a}{1}x^u$, and $\psi_1$, $\psi_{1'}$ injections determined by the localizations for $B_R$ and $B_{Q(R)}$ respectively.\\ 

As $\psi'(S_1)\subseteq S_{1'}$, then $\psi_{1'}(\psi'(S_1))\subseteq \psi_{1'}(S_{1'})\subseteq \left( S^{-1}_{1'}B_{Q(R)}\right)^*$, therefore, by universal property there exists $\varphi$. If $f=\sum a_ux^u \in R_{\q,\sigma}[x_1,\ldots, x_n]$ and $rx^v\in S_1$ then, 

\begin{eqnarray*}
\varphi\left( \frac{f}{rx^v}\right) & = & \varphi\left(\dfrac{\sum a_ux^u}{rx^v} \right)\\
& = & \psi_{1'}(\psi'(rx^v))^{-1}\psi_{1'}\left( \psi'\left( \sum a_ux^u\right) \right)\\
& = &\psi_{1'}\left( \frac{r}{1}x^v\right) ^{-1}\psi_{1'}\left(\sum \frac{a_u}{1}x^u \right)\\
& = & \frac{\frac{1}{1}}{\frac{r}{1}x^v}\frac{\sum\frac{a_u}{1}x^u}{\frac{1}{1}}\\
& = & \frac{\sum\frac{a_u}{1}x^u}{\frac{r}{1}x^v}\\
& =& \frac{\psi'(f)}{\psi'(rx^v)}.
\end{eqnarray*}
Also, $\varphi$ is injective by $\psi'$ and $\psi_{1'}$ are injective.

\end{proof}

Need the following result for the subsequent theorem:

\begin{proposition}\label{divisor}
Let $R$ be a ring and $S\subset R$ a multiplicative subset. If $Q:=S^{-1}R$ exists, then any finite
set $\{q_{1},\ldots, q_{n}\}$ of elements of $Q$ posses a common denominator, i.e., there exists
$r_{1},\ldots,r_{n}\in R$ and $s\in S$ such that $q_{i}=\frac{r_i}{s}, 1\leq i\leq n$.
\end{proposition}
\begin{proof}
See \cite{McConnell}, Lemma 2.1.8.
\end{proof}

\begin{theorem}\label{iso}
Let $R$ be a left Ore domain, then $Q(\Qnn(R))\cong Q(\Qnnq(Q(R))$.
\end{theorem}
\begin{proof}
With the notation of the proof in the Proposition \ref{injection} consider the following diagram of ring homomorphisms
$$
\xymatrix{
S^{-1}_1B_R  \ar@{->}[d]_{\varphi} \ar[r]^{\psi_2} & Q(S^{-1}_1B_R) \ar@{-->}[d]_{\varphi '}\\ 
S^{-1}_{1'}B_{Q(R)}\ar[r]^{\psi_{2'}} & Q(S^{-1}_{1'}B_{Q(R)}) 
}
$$
where $\psi_2$, $\psi_{2'}$ are injections  determined by the localizations of $S_1^{-1}B_R$ and $S_{1'}^{-1}B_{Q(R)}$ respectively and $\varphi$ the  injection of the  Proposition \ref{injection}.\\

By Remark \ref{pctq}, $S_1^{-1}B_R$ and $S_{1'}^{-1}B_{Q(R)}$ are domain, now, if $\frac{p_1}{q_1},\frac{p_2}{q_2}\in S_1^{-1}B_R$ with $\frac{p_1}{q_1}\neq0$, then $p_1\neq0$ and there exist $f_1\neq0$ and $f_2\in B_R$ such that $f_1p_1=f_2p_2$. Then, $\frac{f_1q_1}{1}\frac{p_1}{q_1}=\frac{f_1p_1}{1}=\frac{f_2q_2}{1}=\frac{f_2q_2}{1}\frac{p_2}{q_2}\neq0$, therefore $S_1^{-1}B_R$ is a Ore domain, similarly it has to $S_{1'}^{-1}B_{Q(R)}$. Thus, if $S_2=S_1^{-1}B_R-\{0\}$ and $S_{2'}=S_{1'}^{-1}B_{Q(R)}-\{0\}$ as $\varphi(S_2)\subseteq S_{2'}$, then $\psi_{2'}(\varphi(S_2))\subseteq \psi_{2'}(S_{2'})\subseteq \left( Q(S^{-1}_{1'}B_{Q(R)})\right)^*$, hence, by  universal property there exists $\varphi'$ injective ring homomorphism.\\

Note that if $f,g\in B_R$ and $ax^u, bx^b \in S_1$, then
$$\frac{\frac{f}{ax^u}}{\frac{g}{bx^v}}=\left(\frac{g}{bx^v} \right)^{-1} \frac{f}{ax^u}=\frac{bx^v}{g}\frac{f}{ax^u}= \frac{f'}{g'}$$ and $$\frac{f'}{g'}=\frac{1}{g'}\frac{f'}{1}=\left( \frac{g'}{1}\right) ^{-1}\frac{f'}{1}=\frac{\frac{f'}{1}}{\frac{g'}{1}},$$
where $f', g' \in B_R$ by Remark \ref{pctq} with $r=0$. Similarly is obtained for  $Q(S^{-1}_{1'}B_Q(R))$.\\

Therefore, 
\begin{eqnarray*}
\varphi'\left( \frac{f}{g}\right) & = &\psi_{2'}\left( \varphi\left( \frac{g}{1}\right) \right) ^{-1}\psi_{2'}\left( \varphi\left( \frac{f}{1}\right) \right) \\
& =& \psi_{2'}\left(  \frac{\psi'(g)}{\frac{1}{1}}\right)  ^{-1}\psi_{2'}\left( \frac{\psi'(f)}{\frac{1}{1}} \right)\\
& = & \frac{\frac{1}{1}}{\psi'(g)} \frac{\psi'(f)}{\frac{1}{1}}\\
& = &  \frac{\psi'(f)}{\psi'(g)}.
\end{eqnarray*}

Now, if $f , 0\neq g\in S'_{1'}B_{Q(R)}$, applying Theorem \ref{divisor} must be
\begin{eqnarray*}
\frac{f}{g}&=&\frac{\sum \frac{a_u}{b_u}x^u}{\sum \frac{c_v}{d_v}x^v}=\frac{\frac{1}{s}\sum\frac{a'_u}{1}x^u}{\frac{1}{s'}\sum\frac{c'_v}{1}x^v}=
\left( \sum\frac{c'_v}{1}x^v\right)^{-1} \left( \frac{1}{s'}\right)^{-1} \frac{1}{s}\sum\frac{a'_u}{1}x^u\\
&=& \left( \sum\frac{c'_v}{1}x^v\right)^{-1} \left( \frac{s'}{1} \frac{1}{s}\right) \sum\frac{a'_u}{1}x^u= \left( \sum\frac{c'_v}{1}x^v\right)^{-1} \left( \frac{r'}{r} \right) \sum\frac{a'_u}{1}x^u\\
& = & \left( \sum\frac{c'_v}{1}x^v\right)^{-1} \left( \frac{1}{r}\frac{r'}{1} \right) \sum\frac{a'_u}{1}x^u 
= \left(  \frac{r}{1}\sum\frac{c'_v}{1}x^v\right)^{-1} \left(\frac{r'}{1}  \sum\frac{a'_u}{1}x^u \right)\\
&=&\left(\sum\frac{rc'_v}{1}x^v\right)^{-1} \left( \sum\frac{r'a'_u}{1}x^u \right)\\
&=& \frac{\sum\frac{r'a'_u}{1}x^u}{\sum\frac{rc'_v}{1}x^v}=\frac{\psi'(f')}{\psi'(g')}\\
&=&\varphi\left(\frac{f'}{g'} \right). 
\end{eqnarray*}
where $f'=\sum (r'a'_u)x^u$ y $g'=\sum(rc'_v)x^v$, then $\varphi$ is surjective. Hence $Q(\Qnn(R))\cong Q(\Qnnq(Q(R)))$.

\end{proof}

\subsection{Valuations of skew quantum polynomials.}\label{definitionV}







\begin{theorem}\label{ValPct}
Let $R$ be a left Ore domain and $\nu:Q(\Qnn(R))^*\rightarrow \Gamma$ is a valuation with $\nu(Q(R)^*)=0$, then $\Gamma$ is Abelian.
\end{theorem}
\begin{proof}
$Q(R)$ is a division ring and $Q(\Qnn(R))\cong Q(\Qnnq(Q(R)))$, by Theorem \ref{vs}. $\Gamma$ is Abelian.
\end{proof}

\begin{corollary}
Let $R$ be a left Ore domain, $\nu:Q(\Qnn(R))^*\rightarrow \Gamma$ a valuation with $\nu(Q(R)^*)=0$ and $\Qnnq(Q(R))$ generic, then $\Gamma$ is Abelian.
\end{corollary}

\begin{theorem}\label{v3}
Let $R$ be a left Ore domain, a valuation $\nu:Q(\Qnn(R))^*\rightarrow \Gamma$ with $\nu(Q(R)^*)=0$ and $\Qnnq(Q(R))$ generic. The valuation $\nu$ has maximal rank if only if $\Gamma\cong\Z^{n}$.
\end{theorem}
\begin{proof}
By Theorem \ref{iso}. $Q(\Qnn(R))\cong Q(\Qnnq(Q(R)))$ with $Q(R)$ a division ring, by Theorem \ref{maximal} the valuation $\nu$ has maximal rank if only if $\Gamma\cong\Z^{n}$.
\end{proof}

\subsection{Valuations of skew $PBW$ extension.}

\begin{theorem}
Let $A=\sigma(R)\left\langle x_1,\ldots,x_n \right\rangle $ be a bijective and quasi-commutative  skew $PBW$ extension of a ring $R$. If $R$ is a left Ore domain and $\nu:Q(A)^*\rightarrow \Gamma$ a valuation with $\nu(Q(R)^*)=0$, then $\Gamma$ is Abelian
\end{theorem}

\begin{proof}
By Theorem \ref{oreccb} $A$ is an Ore domain then, $Q(A)$ exists and is a division ring, by Remark \ref{pctq}. $Q(A)\cong Q(\Qrn(R))$ (in particular $r=0$) and by Theorem \ref{ValPct} $\Gamma$ is abelian.
\end{proof}

\begin{corollary}
Let $A$ be a bijective skew $PBW$ extension of a ring $R$. If $R$ is a left Ore domain and $\nu:Q(Gr(A))^*\rightarrow \Gamma$ a valuation with $\nu(Q(R)^*)=0$, then $\Gamma$ is Abelian.
\end{corollary}
\begin{proof}
By Theorem $\ref{filtrado}$ $Gr(A)$ is bijective and quasi-commutative.
\end{proof}


\end{document}